\documentclass[a4paper,11pt]{amsart}
\usepackage{aliascnt}
\usepackage{hyperref}
\usepackage{amsfonts}
\usepackage{mathrsfs}
\usepackage{amsmath}
\usepackage{amsmath,graphics}
\usepackage{amssymb}
\usepackage{fullpage}
\usepackage{indentfirst,latexsym,bm,amsmath,pstricks,amssymb,amsthm,graphicx,fancyhdr,float,color}
\usepackage[all]{xy}
\usepackage{amsthm}
\usepackage{amscd}
\usepackage{hyperref}
\usepackage[all]{hypcap}
\newtheorem{theorem}{Theorem}[section]

\newaliascnt{lemma}{theorem}
\newtheorem{lemma}[lemma]{Lemma}
\aliascntresetthe{lemma}

\newaliascnt{conjecture}{theorem}
\newtheorem{conjecture}[conjecture]{Conjecture}
\aliascntresetthe{conjecture}

\newaliascnt{proposition}{theorem}

\aliascntresetthe{proposition}

\newaliascnt{corollary}{theorem}
\newtheorem{corollary}[corollary]{Corollary}
\aliascntresetthe{corollary}

\newaliascnt{problem}{theorem}

\aliascntresetthe{problem}

\theoremstyle{definition}

\newaliascnt{definition}{theorem}
\newtheorem{definition}[definition]{Definition}
\aliascntresetthe{definition}

\newaliascnt{example}{theorem}
\newtheorem{example}[example]{Example}
\aliascntresetthe{example}

\theoremstyle{remark}

\newaliascnt{remark}{theorem}
\newtheorem{remark}[remark]{Remark}
\aliascntresetthe{remark}

\newaliascnt{remarks}{theorem}

\aliascntresetthe{remarks}

\numberwithin{equation}{subsection}

\def\lra{\longrightarrow}

\def\({$($}
\def\){$)$}

\newcommand{\adele}{{{\mathbb{A}}_f}}

\newcommand{\Qbb}{\mathbb{Q}}
\newcommand{\Rbb}{\mathbb{R}}
\newcommand{\Zbb}{\mathbb{Z}}

\newcommand{\Mcal}{\mathcal{M}}
\newcommand{\Acal}{\mathcal{A}}

\newcommand{\Cbb}{\mathbb{C}}
\newcommand{\la}{\leftarrow}
\newcommand{\ra}{\rightarrow}
\newcommand{\mono}{\hookrightarrow}
\newcommand{\Tcal}{\mathcal{T}}
\newcommand{\GSp}{\mathrm{GSp}}
\newcommand{\Sp}{\mathrm{Sp}}
\newcommand{\bsh}{\backslash}
\newcommand{\isom}{\simeq}

\newcommand{\Gbb}{\mathbb{G}}
\newcommand{\mrm}{\mathrm{m}}
\newcommand{\Gbf}{\mathbf{G}}
\newcommand{\Hbf}{\mathbf{H}}

\newcommand{\Sbb}{\mathbb{S}}
\newcommand{\Res}{\mathrm{Res}}
\newcommand{\ad}{\mathrm{ad}}
\newcommand{\Ad}{\mathrm{Ad}}
\newcommand{\GL}{\mathrm{GL}}
\newcommand{\gfrak}{\mathfrak{g}}
\newcommand{\Lie}{\mathrm{Lie}}
\newcommand{\SL}{\mathrm{SL}}

\newcommand{\der}{\mathrm{der}}
\newcommand{\inv}{{-1}}

\newcommand{\Ker}{\mathrm{Ker}}

\newcommand{\Acalbar}{\overline{\Acal}}

\newcommand{\cosg}{compact open subgroup\ }

\newcommand{\Mbar}{{\overline{M}}}

\newcommand{\Tbf}{{\mathbf{T}}}

\newcommand{\Hom}{\mathrm{Hom}}

\newcommand{\Gr}{\mathrm{Gr}}

\newcommand{\Hcal}{\mathcal{H}}

\newcommand{\sing}{\mathrm{sing}}
\newcommand{\dec}{\mathrm{dec}}
\newcommand{\THcal}{\mathcal{TH}}

\newcommand{\BB}{\mathrm{BB}}

\title{A note on Shimura subvarieties in the hyperelliptic Torelli locus}

\author{Ke CHEN}

\address{School of mathematics, University of Science and Technology of China, Hefei, China, 230026}
\email{kechen@ustc.edu.cn}

\author{Xin Lu}
\address{Institut f\"ur Mathematik, Universit\"at Mainz, Mainz, Germany, 55099}
\email{luxin001@uni-mainz.de}

\author{Kang Zuo}
\address{Institut f\"ur Mathematik, Universit\"at Mainz, Mainz, Germany, 55099}
\email{zuok@uni-mainz.de}


\begin{document}

\begin{abstract}

In this note we prove the non-existence of Shimura subvarieties of positive dimension contained generically in the hyperelliptic Torelli locus for curves of genus $g>7$, which is an analogue of Oort's conjecture in the hyperelliptic case.
\end{abstract}

\maketitle

\tableofcontents

\section{Introduction}\label{sec-introduction}
Let $\Mcal_g$ (resp. $\Acal_g$) be the fine moduli scheme of smooth projective curves of genus $g$ (resp. of principally polarized abelian varieties of dimension $g$) with  level-$N$ structures, $N$ being a fixed integer at least 3 so that the corresponding moduli problems are representable. We have the Torelli map $j^\circ:\Mcal_g\ra\Acal_g$, whose image $\Tcal_g^\circ$ is called the open Torelli locus. The closure $\Tcal_g$ of $\Tcal_g^\circ$ is called the Torelli locus, and $\Tcal_g^\circ$ is known to be an open subscheme of $\Tcal_g$. Note that $\Acal_g$ is a connected Shimura variety, in which we can talk about Shimura subvarieties (cf. \autoref{sec-preliminaries}). A Shimura subvariety $M\subset \Acal_g$ is said to be contained generically in $\Tcal_g$ if $M\subset\Tcal_g$ and $M\cap\Tcal_g^\circ\neq\emptyset$. It was conjectured that:

\begin{conjecture}[Oort]\label{Oort conjecture}
For $g$ sufficiently large, the Torelli locus $\Tcal_g$ contains NO Shimura subvarieties of positive dimension generically.
\end{conjecture}

We refer to the recent survey \cite{moonen-oort-13} of Moonen-Oort and the references there for the history, motivation, applications and further discussion of this conjecture. There has been much progress towards the above conjecture, see for example \cite{chen-lu-zuo,dejong-noot-91,dejong-zhang-07,grushevsky-moller-13,hain-99,lu-zuo-14,moonen-10}, etc.

Inside $\Tcal_g$ there is the hyperelliptic Torelli locus $\THcal_g$ corresponding to Jacobians of hyperelliptic curves (including the non-smooth ones) with $\THcal_g^\circ:=\THcal_g\cap\Tcal_g^\circ=j^\circ(\Hcal_g)$ open in $\THcal_g$, where $\Hcal_g\subset \Mcal_g$ is the locus of smooth hyperelliptic curves. In this paper we study the following hyperelliptic analogue of Oort's conjecture:

\begin{theorem}[hyperelliptic Oort conjecture]\label{hyperelliptic Oort conjecture}
For $g>7$, the hyperelliptic Torelli locus $\THcal_g$ does not contain any Shimura subvariety of positive dimension generically.
\end{theorem}
Similar to the Torelli case, here a Shimura subvariety $M$ of $\Acal_g$ is contained generically in $\THcal_g$ if and only if $M$ is contained in $\THcal_g$ and the intersection $M\cap\THcal_g^\circ$ is non-empty. It is known that when $g$ is small, there indeed exist  Shimura subvarieties of positive dimension contained generically in the hyperelliptic Torelli locus, see for instance \cite{grushevsky-moller-13,moonen-10,lu-zuo-14}. In particular, Grushevsky and M\"oller constructed in \cite{grushevsky-moller-13} infinitely many Shimura curves contained in $\THcal_3$.

%

Assuming the Andr\'e-Oort conjecture for $\Acal_g$, we deduce from the theorem above the following finiteness result on  CM points in the open Torelli locus $\THcal_g^\circ$.
\begin{corollary}
For $g>7$, if the Andr\'e-Oort conjecture for $\Acal_g$ is true, then there exists at most finitely many smooth hyperelliptic   curves of genus $g$ (up to isomorphism) with complex multiplication.
\end{corollary} Coleman's conjecture (cf. \cite{coleman-87}) predicts that for $g$ sufficiently large, there exists at most finitely many smooth curves of genus $g$ (up to isomorphism) whose Jacobians are CM abelian varieties. The corollary above gives a partial answer to the hyperelliptic analogue of this conjecture.


The main idea of the proof is as follows:
\begin{enumerate}
\item[Step 1]
We reduce the problem to the case when $M\subset \Acal_g$ is a simple Shimura variety, in the sense that it is defined by a connected Shimura datum $(\Gbf,X;X^+)$ with $\Gbf^\der$ a $\Qbb$-simple semi-simple $\Qbb$-group. The case of Shimura curves have been studied in \cite{lu-zuo-14}, and we assume that $\Gbf^\der$ is not isomorphic to $\SL_2$ over $\Qbb$, hence the boundary components in the Baily-Borel compactification of $M$ are of codimension at least 2. In particular, the closure $\Mbar$ of $M$ in the Baily-Borel compactification of $\Acal_g$ is obtained by joining boundary components of codimension at least 2, using functorial properties of the Baily-Borel compactification.

\item[Step 2]
Assume that $M$ is a Shimura subvariety of $\Acal_g$ contained generically in $\THcal_g$. Let $C$ be a generic curve in the closure $\Mbar$ of $M$ as above. Then we may take $C$ meeting $\Mbar\bsh M$ trivially due to the codimension condition in Step 1. If $\THcal_g^\sing:=\THcal_g\setminus\THcal_g^\circ$ meets $\Mbar$ also in codimension at least 2, then we may take $C$ meeting $\THcal_g^\sing$ trivially, which contradicts the affineness of the hyperelliptic Torelli locus. Hence the intersection $\THcal_g^\sing\cap M$ contains a divisor of $M$.

\item[Step 3]
The locus of decomposable principally polarized abelian varieties $\Acal^\dec_g$ is a finite union of Shimura subvarieties of $\Acal_g$, and $\Acal_g^\dec\cap\THcal_g=\THcal_g^\sing$. Hence the intersection $\THcal_g^\sing\cap M$ contains a divisor $M'$ which is also a Shimura subvariety. We may then apply dimensional induction to $M'$, and use the result of \cite{lu-zuo-14} to obtain the bound $g>7$.
\end{enumerate}

In \autoref{sec-preliminaries} we collect briefly some facts about Shimura subvarieties,
part of which is reproduced from \cite{chen-lu-zuo}. In \autoref{sec-proof} we prove the main result by completing Step 2 and Step 3 introduced above.

\subsection*{Acknowledgement}This work is supported by SFB/Transregio 45 Periods, Moduli Spaces and Arithmetic of Algebraic Varieties of the DFG (Deutsche Forschungsgemeinschaft), and also supported by National Key Basic Research Program of China (Grant No. 2013CB834202). The first named author is partially supported by National Natural Science Foundation of China (Grant No. 11301495).

\subsection*{Convention and notations} Denote by $\Sbb$ the Deligne torus $\Res_{\Cbb/\Rbb}\Gbb_{\mrm,\Cbb}$. For $k$ a commutative ring, linear $k$-groups stand for affine algebraic $k$-groups. For $\Gbf$ a linear $\Qbb$-group, write $\Gbf(\Rbb)^+$ for the neutral connected component of the Lie group $\Gbf(\Rbb)$, and $\Gbf(\Qbb)^+$ for the intersection $\Gbf(\Qbb)\cap\Gbf(\Rbb)^+$.

\section{Preliminaries on Shimura varieties}\label{sec-preliminaries}
In this section we recall some  facts about Shimura (sub)varieties, functorial properties of Baily-Borel compactification, and the notion of decomposable locus in $\Acal_g$.

We follow \cite{chen-lu-zuo,lu-zuo-14} closely for the basic notions of connected Shimura data and Shimura subvarieties.

\begin{definition}[connected Shimura data, cf. \cite{deligne-pspm}, \cite{milne-05}]
(1) A Shimura datum is a pair   $(\Gbf,X)$ subject to the following constraints:
\begin{enumerate}
\item[SD1]
$\Gbf$ is a connected reductive $\Qbb$-group, and $X$ is a $\Gbf(\Rbb)$-orbit in $\Hom_{\Rbb-\Gr}(\Sbb,\Gbf_\Rbb)$ with $\Sbb=\Res_{\Cbb/\Rbb}\Gbb_{\mrm \Cbb}$ the Deligne torus. We also require that $\Gbf^\ad$ admits no compact $\Qbb$-factors.

\item[SD2]
For any $x\in X$, the composition $\Ad\circ x:\Sbb\ra\Gbf_\Rbb\ra\GL_{\gfrak,\Rbb}$ induces on $\gfrak=\Lie\Gbf$ a rational Hodge structure of type $\{(-1,1),(0,0),(1,-1)\}$.

\item[SD3]
For any $x\in X$, the conjugation by $x(\sqrt{-1})$ induces a Cartan involution on $\Gbf^\ad_\Rbb$.
\end{enumerate}

It is known that $X$ is a finite union of Hermitian symmetric domains, each connected component of which is homogeneous under $\Gbf^\der(\Rbb)^+$.

A morphism between Shimura data is a pair $(f,f_*):(\Gbf,X)\ra(\Gbf',X')$ where $f:\Gbf\ra\Gbf'$ is a $\Qbb$-group homomorphism, such that the push-forward $f_*:\Hom_{\Rbb-\Gr}(\Sbb,\Gbf_\Rbb)\ra\Hom_{\Rbb-\Gr}(\Sbb,\Gbf'_\Rbb)$ sends $X$ into $X'$. When $f$ is an inclusion of a $\Qbb$-subgroup, the push-forward $f_*$ is injective, and we get the notion of Shimura subdata.

In particular, if $(f,f_*):(\Gbf,X)\ra(\Gbf',X')$ is a morphism of Shimura data, then it is easily verified that the pair $(f(\Gbf),f_*(X))$ is a subdatum of $(\Gbf',X')$, called the image subdatum of the morphism.

When $(\Gbf,X)$ is a subdatum of $(\Gbf',X')$ with $\Gbf$ a $\Qbb$-torus, then $X$ consists of a single point $\{x\}$ and one writes $(\Gbf,x)$ for simplicity.


(2) A connected Shimura datum is of the form $(\Gbf,X;X^+)$ where $(\Gbf,X)$ is a Shimura datum and $X^+$ is a connected component of $X$. Notions like morphisms between connected Shimura data, connected Shimura subdata, etc. are defined in the evident way.

\end{definition}


\begin{definition}[Shimura varieties and Shimura subvarieties]
(1)  A (connected) Shimura variety is a quotient of the form $\Gamma\bsh X^+$ where $X^+$ is a connected component from some connected Shimura datum $(\Gbf,X;X^+)$ and $\Gamma\subset\Gbf^\der(\Rbb)^+$ is an arithmetic subgroup. We write $\wp_\Gamma:X^+\ra\Gamma\bsh X^+$ for the uniformization map $x\mapsto \Gamma x$.

(2) For $M=\Gamma\bsh X^+$ a Shimura variety as in (1), a Shimura subvariety of $M$ is of the form $\wp_\Gamma(X'^+)$ where $X'^+$ comes from some connected Shimura subdatum $(\Gbf',X';X'^+)\subset(\Gbf,X;X^+)$. If we choose an arithmetic subgroup $\Gamma'$ of $\Gbf'^\der(\Rbb)^+$ which is also contained in $\Gamma$, then $\wp_\Gamma(X'^+)$ is the same as the image of $\Gamma'\bsh X'^+\ra\Gamma\bsh X^+, \ \Gamma'x'\mapsto \Gamma x'$.

In particular, if $(\Tbf,x)$ is a connected subdatum with $\Tbf$ a $\Qbb$-torus in $\Gbf$, the Shimura subvariety we obtained is a point. In the literature it is often referred to as special points or CM points, because in the Siegel case, i.e. when $(\Gbf,X)=(\GSp_V,\Hcal_V)$ and $M=\Acal_g$ cf.\autoref{Siegel modular varieties} below, they correspond to CM abelian varieties via the modular interpretation.
\end{definition}

\begin{remark}
In standard references on Shimura varieties, like \cite{deligne-pspm} and \cite{milne-05}, a complex Shimura variety is defined adelically as the double quotient $M_K(\Gbf,X)=\Gbf(\Qbb)\bsh(X\times\Gbf(\adele)/K)$ where $(\Gbf,X)$ is a Shimura datum and $K\subset\Gbf(\adele)$ is a \cosg.
{Fix a connected component $X^+$ of $X$ and a set of representatives $\{q\}$ of the double quotient $\Gbf(\Qbb)_+\bsh\Gbf(\adele)/K$ with $\Gbf(\Qbb)_+$ being the stabilizer of $X^+$ in $\Gbf(\Qbb)$,} we get an isomorphism $M_K(\Gbf,Y)\isom\coprod_q\Gamma_K(q)\bsh X^+$, with $\Gamma_K(q):=\Gbf(\Qbb)_+\cap qKq^\inv$ a congruence subgroup of $\Gbf(\Qbb)_+$ which acts on $X^+$ through its image in $\Gbf^\ad(\Qbb)^+$, which in turn is an arithmetic subgroup of $\Gbf^\ad(\Qbb)^+$ by \cite{borel-69} 8.9 and 8.11.
The adelic setting is convenient for discussion of arithmetic properties like canonical models. However in our study it suffices to treat Shimura varieties as complex algebraic varieties, and from the viewpoint of Baily-Borel compactification the definition of connected Shimura varieties as $\Gamma\bsh X^+$ given above is sufficient, because an arithmetic subgroup $\Gamma$ of $\Gbf^\der(\Rbb)^+$ acts on $X^+$ through its image in $\Gbf^\ad(\Rbb)^+$ which is again an arithmetic subgroup of $\Gbf^\ad(\Rbb)^+$ by \cite{borel-69}.\end{remark}

\begin{example}[{Siegel modular varieties, cf. \cite[Example\,2.1.7]{chen-lu-zuo}}]\label{Siegel modular varieties} Let $(V_\Zbb,\psi)$ be a symplectic space over $\Zbb$ with $\psi:V_\Zbb\times V_\Zbb\ra \Zbb$ an symplectic pairing of discriminant $\pm1$. Writing $(V,\psi)$ for the symplectic $\Qbb$-space obtained by base change $\Zbb\ra\Qbb$, we get the connected reductive $\Qbb$-group of simplectic similitude $\GSp_V$ together with a homomorphism of $\Qbb$-groups $\lambda:\GSp_V\ra\Gbb_\mrm$, such that $\psi(gv,gv')=\lambda(g)\psi(v,v')$ for $g\in\GSp_V$ and $v,v'\in V$.

We put $\Hcal_V$ for the set of polarizations of $(V,\psi)$, i.e., the set of $\Rbb$-group homomorphisms $h:\Sbb\ra\GSp_{V,\Rbb}$ such that $h$ induces a $\Cbb$-structure on $V_\Rbb$ and $\psi(h(\sqrt{-1})v,v')$ is symmetric and definite (positive or negative). The set $\Hcal_V$ is naturally identified with the Siegel double half-space of genus $g=\frac{1}{2}\dim_\Qbb V $, and the pair $(\GSp_V,\Hcal_V)$ is a Shimura datum. Let $\Hcal_V^+$ be the connected component of $\Hcal_V$ corresponding to positive definite polarizations, and let $\Gamma\subset\Sp_V(\Rbb)$ be an arithmetic subgroup. Then the quotient $\Gamma\bsh \Hcal_V^+$ is referred to as a Siegel modular variety of level $\Gamma$. This is motivated by the case when $\Gamma=\Gamma(N)=\Ker(\Sp_{V_\Zbb}(\Zbb)\ra\Sp_{V_\Zbb}(\Zbb/N))$ is the $N$-th principal congruence subgroup using the integral structure $V_\Zbb$, which gives $\Gamma(N)\bsh\Hcal_V^+$ as the moduli space of principally polarized abelian varieties of dimension $g$ with level-$N$ structure, constructed by Mumford in \cite{fogarty-kirwan-mumford}.

As we have mentioned, in this paper we fix $N\geq 3$ and put $\Acal_g=\Acal_{g,N}$ to be the Siegel modular variety associated to the standard symplectic space on $V_\Zbb=\Zbb^{2g}$. The condition $N\geq 3$ assures the representability of the moduli problem. The Shimura datum is also written as $(\GSp_{2g},\Hcal_g)$.


\end{example}

For simplicity we only consider arithmetic subgroups that are torsion-free. The quotients $\Gamma\bsh X^+$ are therefore smooth complex manifolds.

\begin{theorem}[Baily-Borel compactification, \cite{baily-borel-66}. \cite{borel-metric}]\label{baily-borel-66 compactification}
Let $M=\Gamma\bsh X^+$ be a Shimura variety. Then the following hold:

(1) $M$ is a normal quasi-projective algebraic variety over $\Cbb$, and it admits a compactification, called the Baily-Borel compactification $M^\BB$, which is universal in the sense that if $M\ra Z$ is a morphism of complex algebraic varieties with $Z$ projective, then it admits a unique factorization $M\mono M^\BB\ra Z$.

(2) The boundary components of $M$, i.e., irreducible components of $M^\BB \setminus M$ are of codimension at least 2, unless $\Gbf^\der$ admits a $\Qbb$-factor isogeneous to $\SL_{2,\Qbb}$.

(3) The Baily-Borel compactification is functorial, in the sense that if $f:(\Gbf',X';X'^+)\ra(\Gbf,X;X^+)$ is a morphism of connected Shimura data and $\Gamma'\subset\Gbf'^\der(\Rbb)^+$ and $\Gamma\subset\Gbf^\der(\Rbb)^+$ are arithmetic subgroups such that $f(\Gamma')\subset\Gamma$, then the induced map $f:M'=\Gamma'\bsh X'^+\ra M=\Gamma\bsh X^+$ is a morphism between algebraic varieties over $\Cbb$, and it extends uniquely to the compactifications $M'^\BB\ra M^\BB$.
\end{theorem}

\begin{corollary}[boundary components]\label{boundary components}
Let $M=\Gamma\bsh X^+$ be a Shimura variety defined by $(\Gbf,X;X^+)$ and an arithmetic subgroup $\Gamma\subset\Gbf^\der(\Rbb)^+$. Let $M'\subset M$ be a Shimura subvariety defined by $(\Gbf',X';X'^+)\subset(\Gbf,X;X^+)$, and assume that $\Gbf'^\der$ admits no $\Qbb$-factor isogeneous to $\SL_{2,\Qbb}$. Write $\Mbar'$ for the closure of $M'$ inside $M^\BB$,
then the irreducible components of $\Mbar'\setminus M'$ are of codimension at least 2 in $\Mbar'$.
\end{corollary}

\begin{proof}
The inclusion $M'\mono M$ is a morphism of algebraic varieties $\Gamma'\bsh X'^+\ra \Gamma\bsh X^+$ with $\Gamma'=\Gamma\cap\Gbf'^\der(\Rbb)^+$ a torsion-free arithmetic subgroup of $\Gbf'^\der(\Rbb)^+$, which is also a closed immersion. By \autoref{baily-borel-66 compactification}(3) it extends to a morphism of projective algebraic varieties $M'^\BB\ra M^\BB$ which is generically injective, and the closure $\overline{M}'$ of $M'$ in $M^\BB$ is also the closure of the image of $M'^\BB$. Since $M'^\BB$ only joins to $M'$ finitely many boundary components of codimension at least 2, we see that $\overline{M}'$ only differs from $M'$ by a closed subvariety of codimension at least 2.\end{proof}

\begin{definition}[decomposable locus]
A principal polarized abelian variety $A$ over $\Cbb$ is said to be decomposable if it is isomorphic to a product $A=A_1\times A_2$ with $A_1$ and $A_2$ both principally polarized of dimension $>0$ such that the polarization of $A$ is isomorphic to the one induced by the two polarizations on $A_1$ and $A_2$ respectively. We thus get the locus $\Acal_g^\dec\subset\Acal_g$ of decomposable principal polarized abelian varieties.
\end{definition}

\begin{example}[Shimura subvarieties of decomposable locus]\label{Shimura subvarieties of decomposable locus}

Given $(U,\psi_U)$ and $(W,\psi_W)$ two symplectic $\Qbb$-spaces of dimension $2m$ and $2n$ respectively, the direct sum $V=U\oplus W$ naturally carries a symplectic structure $$\psi_V:=\psi_U\oplus\psi_W:(u,w)\times(u',w')\mapsto\psi_U(u,u')+\psi_W(w,w').$$ This gives rise to the following $\Qbb$-group homomorphism $$f_{U,W}:\GSp_{U,W}:=\GSp_U\times_{\Gbb_\mrm}\GSp_W\ra\GSp_V$$ which is an inclusion: the fibred product $\GSp_{U,W}$ is defined by the two homomorphisms $$\lambda_U:\GSp_U\ra\Gbb_\mrm\la\GSp_W:\lambda_W$$ and it is the $\Qbb$-subgroup of $\GSp_V$ whose elements can be written as pairs $(g_U,g_W)$ with $g_U\in\GSp_U$ and $g_W\in\GSp_W$ acting on $U$ and on $W$ respectively with the same scalar of similitude $\lambda_V((g_U,g_W))=\lambda_U(g_U)=\lambda_W(g_W)$.

We proceed to show that the $\Qbb$-group homomorphism $f_{U,W}$ above extends to a morphism of Shimura data $(\GSp_{U,W},\Hcal_{U,W})\ra(\GSp_V,\Hcal_V)$. Here $\Hcal_{U,W}$ is the set of pairs $(h_U,h_W)$ in $\Hcal_U\times\Hcal_W$ such that $\lambda_U\circ h_U=\lambda_W\circ h_W$. Such a pair $(h_U,h_W)$ naturally defines an element of $\Hcal_V$. The set $\Hcal_{U,W}$ is  homogeneous under $\GSp_{U,W}(\Rbb)$. In fact, given two such pairs $(h_U,h_W)$ and $(h'_U,h'_W)$, there exists some element $g_U\in\GSp_U(\Rbb)$ such that $g_U(h_U)=h'_U$. Choose any  $g=(g_U,g_W)\in\GSp_{U,W}$ lifting $g_U$, we get a third pair $(h'_U,h''_W)=g(h_U,h_W)=(g_U(h_U),g_W(h_W))$, with $\lambda_W(g_W(h_W))=\lambda_U(g_U(h_U))=\lambda_U(h'_U)=\lambda_W(h'_W)$, hence $g_W(h_W)$ and $h'_W$ only differs by the conjugation of some element of $\Sp_W(\Rbb)=\Ker(\GSp_W(\Rbb)\ra\Gbb_\mrm(\Rbb))$, and there exists $g'=(g_U,g_Wg'_W)\in\GSp_{U,W}(\Rbb)$ with $g'_W\in\Sp_W(\Rbb)$ sending $(h_U,h_W)$ to $(h'_U,h'_W)$.

We verify briefly that the pair $(\GSp_{U,W},\Hcal_{U,W})$ satisfies the axioms defining Shimura data. For the Hodge type given by  $h=(h_U,h_W)$ on $\Lie\GSp_{U,W}$, it suffices to notice that $h:\Sbb\ra\GSp_{V,\Rbb}$ has image in $\GSp_{U,W,\Rbb}$ and thus the adjoint action of $h(\Sbb)$ on $\Lie\GSp_{V,\Rbb}$ stabilizes $\Lie\GSp_{U,W,\Rbb}$. The adjoint quotient $\GSp_{U,W}^\ad$ is clearly isomorphic to $\GSp_U^\ad\times\GSp_W^\ad$, and $h(\sqrt{-1})$ induces a Cartan involution on it because it is so with $h_U$ for $\GSp_U^\ad$ and with $h_W$ for $\GSp_W^\ad$.

In particular, $\GSp_{U,W}^\der\isom\Sp_U\times\Sp_W$, and a connected component of $\Hcal_{U,W}$ is isomorphic to $\Hcal_U^+\times\Hcal_W^+$. This gives us a connected Shimura subdatum $$f_{U,W}:(\GSp_{U,W},\Hcal_{U,W};\Hcal^+_U\times\Hcal^+_W)\mono(\GSp_V,\Hcal_V;\Hcal_V^+).$$

If $(U,\psi_U)$ and $(W,\psi_W)$ are given by standard integral symplectic structures $U_\Zbb\isom\Zbb^{2m}$ and $W_\Zbb\isom\Zbb^{2n}$, then we naturally have $V$ given by the standard integral one $V_\Zbb\isom\Zbb^{2g}$ with $g=m+n$. The $N$-th principal congruence subgroup $\Gamma_V(N)=\Ker(\Sp_{V_\Zbb}(\Zbb\ra\Sp_{V_\Zbb}(\Zbb/N)))$ naturally restricts to the congruence subgroup $\Gamma_U(N)\times\Gamma_W(N)$ of $\GSp^\der_{U,W}(\Rbb)^+$ via $\Sp_U\times\Sp_W=\GSp^\der_{U,W}\mono\Sp_V$, and we get $\Gamma_U(N)\times\Gamma_W(N)\bsh\Hcal_VU^+\times\Hcal_W^+$ as a Shimura subvariety of $\Acal_g=\Gamma_V(N)\bsh\Hcal_V^+$, which we denote as $\Acal_{m,n}$ with $m,n>0$ and $m+n=g$.

\end{example}



\begin{lemma}\label{A-dec} The decomposition locus $\Acal_g^\dec$ is a finite union of Shimura subvarieties in $\Acal_g$.
\end{lemma}

\begin{proof} If $A$ is a principally polarized abelian variety decomposed as $A\isom A_1\times A_2$ with $A_1$ of dimension $m$ and $A_2$ of dimension $n=g-m$, where we assume for simplicity $m\leq n$, then the point in $\Acal_g$ parameterizing $A_1\times A_2$ naturally lies in $\Acal_{m,g-m}\subset\Acal_g$ in the sense of \autoref{Shimura subvarieties of decomposable locus}. Hence we get $$\Acal_g^\dec=\bigcup_{1\leq m\leq\lfloor g/2\rfloor}\Acal_{m,g-m}$$ which is a finite union of Shimura subvarieties. \end{proof}

Finally we mention the following useful fact:

\begin{lemma}[intersection of Shimura subvarieties]\label{intersection of Shimura subvarieties} Let $M'$ and $M''$ be Shimura subvarieties of an ambient Shimura variety $M=\Gamma\bsh X^+$ defined by $(\Gbf,X;X^+)$.  Then $M'\cap M''$ is a finite union of Shimura subvarieties of $\Acal_g$ if the intersection is non-empty.
\end{lemma}

\begin{proof}
Write  $\wp:X^+\ra M$ for the uniformization map. Let $M'$ and $M''$ be defined by connected subdata $(\Gbf',X';X'^+)$ and $(\Gbf'',X'';X''^+)$ respectively. Then the non-empty  intersection $\wp(X'^+)\cap\wp(X''^+)$  implies that the intersection
$$\left(\bigcup_{\gamma\in\Gamma}\gamma X'^+\right)\bigcap\left(\bigcup_{\gamma\in\Gamma}\gamma X''^+\right)$$
is non-empty, and we can find $\gamma\in\Gamma$ such that $X'^+\cap\gamma X''^+\neq\emptyset$. Since $(\Gbf'',X'';X''^+)$ and $(\gamma\Gbf''\gamma^\inv,\gamma X'';\gamma X''^+)$ defines the same Shimura subvariety $M''$, we may assume for simplicity that $X'^+\cap X''^+\neq\emptyset$.

We thus take $x\in X'^+\cap X''^+$. Then the homomorphism $x:\Sbb\ra\Gbf_\Rbb$ factors through $\Hbf_\Rbb$, with $\Hbf$ being the neutral component of the intersection $\Gbf'\cap\Gbf''$. We claim that:\begin{enumerate}
\item[(a)] $\Hbf$ is a reductive $\Qbb$-subgroup of $\Gbf$;

\item[(b)] $\Hbf$ contains a maximal connected reductive $\Qbb$-subgroup $\Hbf'$ such that $\Hbf'^\ad$ admits no compact $\Qbb$-factors, and  $x:\Sbb\ra\Gbf_\Rbb$ factors through $\Hbf'_\Rbb$;

\item[(c)] Write $\Hbf'(\Rbb)x$ for the orbit of $x$ under $\Hbf'(\Rbb)$, then the pair $(\Hbf',\Hbf'(\Rbb)x)$ is a Shimura subdatum of $(\Gbf,X)$, and clearly $(\Hbf',\Hbf'(\Rbb)x;\Hbf'(\Rbb)^+x)$ is a connected Shimura subdatum of $(\Gbf,X;X^+)$.
\end{enumerate}
In particular, (c) implies that if $\wp(x)$ is a point in $M'\cap M''$ then it is a point in the Shimura subvariety defined by some $\Qbb$-group $\Hbf'$ uniquely determined by $\Gbf'$ and $\Gbf''$.
Hence the lemma would follow from \cite[Lemma\,3.7]{ullmo yafaev}, which asserts that for a given reductive $\Qbb$-subgroup $\Hbf'$ of $\Gbf$, there are at most finitely many Shimura subdata of the form $(\Hbf',Y)$ in $(\Gbf,X)$, and similarly for connected Shimura subdata of the form $(\Hbf',Y;Y^+)$ of $(\Gbf,X;X^+)$.

The proofs of the claims (a)-(c) are as follows:

(a) $\Hbf$ is a $\Qbb$-subgroup of $\Gbf$ such that $x(\Sbb)\subset\Hbf_\Rbb$. In particular, the centralizer of $x(\Sbb)$ in $\Gbf_\Rbb$ is compact modulo the center of $\Gbf_\Rbb$.
Applying \cite[Lemma\,5.1]{eskin mozes shah} we see that $\Hbf_\Rbb$ is reductive, hence $\Hbf$ is reductive.

(b) $\Hbf$ admits an almost direct product $\Hbf=\Hbf_0\Hbf_1\Hbf_2$ where $\Hbf_1$ is generated by non-compact $\Qbb$-simple normal semi-simple $\Qbb$-subgroups of $\Hbf$,  $\Hbf_2$ is generated by the compact ones, and $\Hbf_0$ is the connected center. Note that $\Hbf_1\Hbf_2$ equals $\Hbf^\der$, and we put $\Hbf':=\Hbf_0\Hbf_1$. To show that $x(\Sbb)$ is contained in $\Hbf'_\Rbb$, it suffices to show that the intersection $x(\Sbb)\cap\Hbf_{2,\Rbb}$ is zero-dimensional. But the inclusion $x(\Sbb)\subset\Hbf_\Rbb$ implies that the conjugation by $x(\sqrt{-1})$ induces a Cartan involution on $\Hbf^\der_\Rbb=\Hbf_{1,\Rbb}\Hbf_{2,\Rbb}$, which fixes the compact part $\Hbf_{2,\Rbb}$, hence $\Hbf_{2,\Rbb}$ is centralized by $x(\Sbb)$, which is essentially the same arguments used in \cite{ullmo yafaev} (right before Lemma 3.6).

(c) The connected reductive $\Qbb$-group $\Hbf'$ of $\Gbf$ admits no compact semi-simple $\Qbb$-factors. The inclusion $x(\Sbb)\subset\Hbf'_\Rbb$ implies that $\Lie\Hbf'$ is a rational Hodge substructure of $\Lie\Gbf$ as $\Lie\Hbf'_\Rbb$ is stabilized by the adjoint action of $x(\Sbb)$ on $\Lie\Gbf_\Rbb$. Hence the condition on Hodge types and on Cartan involution are both satisfied, and we get a Shimura subdatum $(\Hbf',Y)$ with $Y=\Hbf'(\Rbb)x$ being the orbit of $x$ under $\Hbf'(\Rbb)$ inside $X$. We further have a connected subdatum $(\Hbf',Y;Y^+)$, with $Y^+=\Hbf'(\Rbb)^+x$ the connected component of $Y$ containing $x$, and the Shimura subvariety it defines is contained in $M'$ and $M''$ passing through $\wp(x)$.\end{proof}






\section{Proof of the main result}\label{sec-proof}
As we have explained in \autoref{sec-introduction} we prove the main result by induction on the dimension of a given Shimura subvariety $M$ contained generically $\THcal_g$.
The bound $g>7$ comes from the following theorem proved in \cite[Theorem\,E]{lu-zuo-14}:

\begin{theorem}[Lu-Zuo]\label{Lv-Zuo} For $g>7$, the hyperelliptic Torelli locus $\THcal_g$ does not contain generically any totally geodesic curves of $\Acal_g$.
\end{theorem}
Here   totally geodesic subvarieties (including the one-dimensional case, namely totally geodesic curves) are closed algebraic subvarieties in $\Acal_g$ which are totally geodesic for the K\"ahler structure. Shimura subvarieties are always totally geodesic. See \cite{lu-zuo-14} for further details.

 We start with the following property on non-simple Shimura data.

\begin{lemma}[non-simple Shimura data]\label{non-simple Shimura data} Let $(\Gbf,X;X^+)$ be a connected Shimura datum. Assume the associated datum of adjoint type $(\Gbf^\ad,X^\ad;X^+)$ admits a decomposition $(\Gbf_1,X_1;X_1^+)\times(\Gbf_2,X_2;X_2^+)$ with $\Gbf_i$ non-trivial semi-simple $\Qbb$-groups of adjoint type. Then there exists a subdatum $(\Gbf',X';X'^+)$ of $(\Gbf,X;X^+)$, whose image under $(\Gbf,X;X^+)\ra(\Gbf^\ad,X^\ad;X^+)$ is of the form $(\Gbf_1,X_1;X_1^+)\times(\Tbf_2,x_2)$ where $(\Tbf_2,x_2)\subset(\Gbf_2,X_2;X_2^+)$ is a subdatum of CM type, i.e., $\Tbf_2$ is a $\Qbb$-torus in $\Gbf_2$ and $x_2$ is a point in $X_2 ^+$.
\end{lemma}
\begin{proof}
Take any CM subdatum  $(\Tbf_2,x_2)$ of $(\Gbf_2,X_2;X_2^+)$. The pre-image of $\Gbf_1\times\Tbf_2$ under $\Gbf\ra\Gbf^\ad\isom\Gbf_1\times\Gbf_2$ is a $\Qbb$-subgroup $\Hbf$ of $\Gbf$, which is mapped onto $\Gbf_1\times\Tbf_2$. The kernel of $\Hbf\ra\Gbf\ra\Gbf^\ad$ is central in $\Gbf$, whose connected component is a $\Qbb$-subtorus of the center of $\Gbf$. Hence $\Hbf$ is reductive. Write $\Gbf'$ for the neutral component of $\Hbf$, and take $x=(x_1,x_2)\in X^+\isom X_1^+\times X_2^+$ for some $x_1\in X_1^+$. Viewing $x$ as a point in $X^+$ for the datum $(\Gbf,X;X^+)$, we see that $x(\Sbb)\subset\Hbf_\Rbb$ because when viewing $x_i$ as a point of $X_i^+$ of the datum $(\Gbf_i,X_i;X_i^+)$ we have $x_1(\Sbb)\subset\Gbf_{1,\Rbb}$ and $x_2(\Sbb)\subset\Tbf_{2,\Rbb}$. In particular $x(\Sbb)$ is contained in $\Gbf'_\Rbb$ the neutral component of $\Hbf_\Rbb$.
The $\Qbb$-group $\Gbf'$ is a reductive $\Qbb$-subgroup of $\Gbf$, whose adjoint quotient is $\Gbf_1$, admitting no compact $\Qbb$-factors.

We claim that the pair $(\Gbf',X'=\Gbf'(\Rbb)x)$ is a Shimura subdatum of $(\Gbf,X)$: first of all the action of $\Sbb$ on $\Lie\Gbf'_\Rbb$ by the adjoint action coincides with the action of $\Sbb$ on $\Lie\Gbf_\Rbb$, which stabilizes $\Lie\Gbf'_\Rbb$ because $x(\Sbb)\subset\Gbf'_\Rbb$; the remaining conditions on Hodge types and Cartan involutions are valid for $x$, and clearly invariant when we conjugate $x$ by any element $g\in\Gbf'(\Rbb)$.

Take $X'^+$ to be the connected component of $X'$ containing $x$, we get a connected Shimura subdatum $(\Gbf',X';X'^+)$, whose image in $(\Gbf_1,X_1;X_1^+)\times(\Gbf_2,X_2;X_2^+)$ is clearly $(\Gbf_1,X_1;X_1^+)\times(\Tbf_2,x_2)$.
\end{proof}

We also need the notion of Hecke translation of Shimura subvarieties. Here it suffices to use the following version (cf. \cite[Definition\,2.1.9]{chen-lu-zuo}):

\begin{definition}[Hecke translation]\label{Hecke translation}
Let $M=\Gamma\bsh X^+$ be a connected Shimura variety defined by a connected Shimura datum $(\Gbf,X;X^+)$
and some arithmetic subgroup $\Gamma\subset\Gbf^\der(\Rbb)^+$, with $\wp=\wp_\Gamma:X^+\ra M$ the uniformization map.
Let $M'\subset M$ be the Shimura subvariety $\pi(X'^+)$ defined by some connected subdatum $(\Gbf',X';X'^+)$.
For $q\in\Gbf(\Qbb)^+$, the Hecke translation of $M'$ by $q$ is defined  as $\wp(qX'^+)$,
which is the Shimura subvariety associated to $(q\Gbf'q^\inv,qX';qX'^+)$.

We clearly have the equality $\wp(qX'^+)=\wp(q'X'^+)$   as Shimura subvarieties of $M$ whenever $q=\gamma q'$ for some element $\gamma\in\Gamma$.

The union $\bigcup_{q\in\Gbf(\Qbb)^+}\wp(qX'^+)$ is referred to as the Hecke orbit of $M'$ in $M$.
\end{definition}

\begin{lemma}[density of Hecke orbits]\label{density of Hecke orbits}
Let $M$ be a Shimura subvariety of $\Acal_g$ defined by a connected subdatum $(\Gbf,X;X^+)$, which is contained generically in the hyperelliptic Torelli locus $\THcal_g$.

(1) Assume that $M$ contains a proper Shimura subvariety $M'\subsetneq M$ of dimension $>0$ defined by some subdatum $(\Gbf',X';X'^+)$. Then there exists a Hecke translate $M''=\wp(qX'^+)$ of $M'$ in $M$ contained generically in $\THcal_g$.

(2) It suffices to prove the main theorem for $M$ such that $\Gbf^\der$ is $\Qbb$-simple. 
\end{lemma}

\begin{proof}(1) This is essentially the same as \cite[Lemma\,2.1.10]{chen-lu-zuo}. More precisely, the real approximation for linear $\Qbb$-groups implies the density of $\Gbf(\Qbb)^+$ in $\Gbf(\Rbb)^+$, hence the Hecke orbit $\bigcup_{q\in\Gbf(\Qbb)^+}\wp(qX'^+)$ is dense in $M$. Since $M\subset\THcal_g$ meets $\THcal_g^\circ$ non-trivially, there exists some $q\in\Gbf(\Qbb)^+$ such that $\wp(qX'^+)\subset M$ is contained generically in $\THcal_g^\circ$.

(2) Assume that $M$ is defined by some subdatum $(\Gbf,X;X^+)$ such that $\Gbf^\ad$ is NOT $\Qbb$-simple. Then by \autoref{non-simple Shimura data} $(\Gbf,X;X^+)$ contains some subdatum $(\Gbf',X';X'^+)$ with $0<\dim X'^+<\dim X^+$.
It defines a Shimura subvariety $M'\subset M$ of dimension $>0$.
If $M$ is not contained generically in $\THcal_g$, then after passing to a suitable Hecke translate we may assume that $M'$ is contained generically in $\THcal_g$, and to prove the main theorem it suffices to exclude the generic inclusion of $M'$ in $\THcal_g$.
\end{proof}

The affineness of $\THcal_g^\circ$ plays an essential role in the proof of the main result, similar to its usage in \cite{dejong-zhang-07}. 

\begin{theorem}\label{thm-affine-hyper}
The open hyperelliptic locus $\THcal_g^\circ$ is affine.
\end{theorem}
In fact it is clear that under the Torelli morphism $j^\circ$, $\THcal_g^\circ\subset\Acal_g$ is isomorphic to the hyperelliptic locus $\Hcal_g$ inside $\Mcal_g$, where a level-$N$ structure has been imposed  with $N\geq 3$.

\begin{proof}[Proof of \autoref{hyperelliptic Oort conjecture}]
By \autoref{density of Hecke orbits}, it suffices to exclude simple Shimura varieties $M$ contained generically in $\THcal_g$ for $g>7$.

The one-dimensional case, i.e., the case when $M$ a Shimura curve, is done in \autoref{Lv-Zuo}. We may thus assume that $\dim M\geq 2$. In particular, this implies that $\Gbf^\der$ is NOT isogeneous to $\SL_{2,\Qbb}$, and the closure $\Mbar$ of $M$ in the minimal compactification $\Acalbar_g:=\Acal_g^\BB$ admits no boundary components of codimension one. One may thus take a generic projective curve $C$ in $\Mbar$ which meets $\Mbar\bsh M$ trivially.

If the singular locus in $\THcal_g$, namely the intersection $\THcal_g^\sing:=\THcal_g\cap\Acal_g^\dec$,
also meets in $\Mbar$ in codimension at least 2, then we may further choose $C$  meeting $\THcal_g^\sing$ trivially.
But this would give a projective curve $C$ contained in the open Torelli locus $\THcal_g^\circ$,
which contradicts the affineness of $\THcal_g^\circ$ by \autoref{thm-affine-hyper}.
Hence at least one of the irreducible components in $\THcal_g^\sing\cap M=\Acal_g^\dec\cap M$
is of codimension one in $M$.

Since $\Acal_g^\dec$ is a finite union of Shimura subvarieties in $\Acal_g$ by \autoref{A-dec}, using \autoref{intersection of Shimura subvarieties} we see that $\Acal_g^\dec\cap M$ is a finite union of Shimura subvarieties in $M$, and one of them is of codimension one in $M$.
Repeating the Hecke translation arguments as in \autoref{density of Hecke orbits}, we get a Shimura subvariety $M'\subset M$ of dimension $>0$ contained generically in $\THcal_g$. We may then apply the induction on dimensions.
\end{proof}

\end{document}